\documentclass[11pt]{amsart} \textwidth=14.5cm \oddsidemargin=1cm
\usepackage{setspace}
\usepackage{amsthm}
\usepackage[rgb]{xcolor}
\usepackage{times}
\usepackage[colorlinks=true,linkcolor=blue,citecolor=blue]{hyperref}
\usepackage{mathrsfs}  

\usepackage[all,cmtip]{xy}
\usepackage{tikz}
\usepackage{amscd}
\usepackage{amsmath}
\usepackage{amsfonts}
\usepackage{amsxtra}
\usepackage{amssymb}
\usepackage{eucal}
\usepackage{latexsym,todonotes}
\usepackage{stmaryrd}
\usepackage{graphicx}%
\usepackage{dynkin-diagrams}
\usepackage{tikz}

\usetikzlibrary{intersections}
\newtheorem{theorem}{Theorem}[section]

\newtheorem{corollary}[theorem]{Corollary}

\newtheorem{lemma}[theorem]{Lemma}

\newtheorem{proposition}[theorem]{Proposition}

\numberwithin{equation}{section}

\theoremstyle{remark}
\newtheorem{remark}[theorem]{Remark}

\newcommand{\N}{\mathbb{N}}
\newcommand{\Z}{\mathbb{Z}}

\newcommand{\Q}{\mathbb{Q}}

\newcommand{\cM}{\mathcal{M}}

\newcommand{\hf}{\frac12}

\newcommand{\Hy}{\mathscr{H}}

\newcommand{\HB}{\mathscr H_{B_d}}
\newcommand{\e}{\varepsilon}

\newcommand{\Id}{\text{Id}}

\newcommand{\M}{\mathbb M}
\newcommand{\V}{\mathbb V}

\newcommand{\io}{\imath}
\newcommand{\bu}{\bullet}

\newcommand{\iba}{\psi_{\io}}

\newcommand{\nb}{m}

\newcommand{\I}{\mathbb I}
\newcommand{\Ibw}{\I_{r|m|r}}
\newcommand{\Ianti}{\I_{r|m|r}^{d,-}}
\newcommand{\Iwl}{\mathbb I_{\circ}^{-}}
\newcommand{\Iwr}{\mathbb I_{\circ}^{+}}
\newcommand{\Ib}{\mathbb I_{\bu}}

\newcommand{\fW}{{^fW}}
\newcommand{\Wf}{W_f}

\begin{document}
\title{Quasi-parabolic Kazhdan-Lusztig bases and reflection subgroups}
\author{Zachary Carlini and Yaolong Shen}
\address{Department of Mathematics, University of Virginia, Charlottesville, VA 22904}
\email{zic4zfr@virginia.edu (Carlini), ys8pfr@virginia.edu (Shen)} 


\maketitle
\begin{abstract}
     Recently, Wang and the second author constructed a bar involution and canonical basis for a quasi-permutation module of the  Hecke algebra associated to a type B Weyl group $W$, where the basis is parameterized by  left cosets of a quasi-parabolic reflection subgroup in $W$. In this paper we provide an alternative approach to these constructions, and then generalize to Coxeter groups which contain a product of type B Weyl groups as a parabolic subgroup.
\end{abstract}
 \setcounter{tocdepth}{1}


\section{Introduction}

Let $W_d$ be the Weyl group of type $B_d$ and $\HB$ denote the associated Hecke algebra generated by $H_0, H_1, \ldots, H_{d-1}$ in 2 parameters $q,p$, which contains the Hecke algebra $\Hy_{S_d}$ as a subalgebra. (In the introduction, we shall assume that $p$ is an integer power of $q$; a reader can take $p=q$.)

Then we consider reflection subgroups of $W_d$ of the form 
\begin{align}
  \label{eq:Wf1}
W_f =W_{{m_1}}\times \ldots \times W_{{m_k}}\times S_{m_{k+1}}\times \ldots\times S_{m_l}. \end{align} 
where $m_1+\cdots+m_l=d$, $k\le l$ and all $m_i$ are positive integers. $W_f$ is called a quasi-parabolic subgroup of $W_d$, cf. \cite{SW21}. Clearly $W_f$ is a parabolic subgroup of $W_d$ if and only if $k \le 1$.  

Moreover, for $k\le 1$,  there exists a right $\HB$-module $\M_f$, the induced module from the trivial module of the subalgebra $\Hy(W_f)$, parameterized by the set ${}^fW$ of right minimal length representatives of $W_f$. The celebrated Kazhdan-Lusztig (KL) basis on the regular representation of $\HB$ (see \cite{KL79} for $p=q$, and \cite{Lus03} for $p\in q^\Z$) admits a parabolic generalization in terms of $\M_f$ (see Deodhar  \cite{De87}); that is, $\M_f$ admits a bar involution and a distinguished bar-invariant basis, known as the parabolic KL basis. 

In \cite{SW21}, the authors extended the above classic works of Kazhdan, Lusztig and Deodhar to construct canonical bases (also called quasi-parabolic KL bases)  of type B associated to arbitrary reflection subgroups $W_f\subset W_d$ of the form \eqref{eq:Wf1}. 

Since $W_f$ may not be parabolic in general, the Hecke algebra $\Hy(W_f)$ is not a subalgebra of $\HB$ in any natural manner, and hence 
$\M_f$ is not an induced module from an $\Hy(W_f)$-module in general. However, the explicit formulas in \cite[Proposition 3.8]{SW21} for the actions of the generators $H_i$ on the standard basis of $\M_f$ parameterized by the minimal length coset representatives for $W_f \backslash W_d$ look identical to those for $W_f$ parabolic. 

The goal of this paper is to explore whether the above constructions can be generalized to other quasi-parabolic subgroups of Coxeter groups, and to establish some results in this direction. 

A complete classification of of all reflection subgroups of $W$ was given in \cite{DPR13}. When $W=W(G_2)$ is of type $G_2$, we checked by case-by-case direct computation that the explicit formulas of the actions of the generators of $\Hy$ on the standard basis of $\M_f$ as in \cite[Proposition 3.8]{SW21} indeed define module structures on $\M_f$ for all the reflection subgroups $W_f\subset W(G_2)$; see \S \ref{sec:G2}.  However, this property fails when we try to compute for various reflection subgroups of $W$ when $W$ is of type $F_4$; two of such examples are given in \S \ref{sec:F4fail} and \S \ref{F4:B3A1}.

In \cite[Remark 3.13]{SW21}, Wang and the second author proposed a $3$-step induction reformulation of the type B quasi-permutation modules for which we will give a detailed proof in \S \ref{sec:3step}. Motivated by the $3$-step induction and the computation of special cases, we consider a chain of reflection groups $W_f\subset W_{d}\stackrel{\text{par.}}{\subset} W.$
Since $W_d$ is a parabolic subgroup of $W$, we have naturally an induced module
$
\mathcal M=\M_f\otimes_{\HB}\Hy
$ 
parameterized by the minimal length coset representatives for $W_f\backslash W$. We also further generalize the construction of the bar map and the (quasi-parabolic KL) canonical basis on $\M_f$ to the level of $\mathcal M$.

This paper is organized as follows. In \S \ref{sec:minimallength} we study ${}^fW_d$ (resp. ${}^dW,\ {}^fW$), the set of minimal length right coset representatives of $W_f$ in $W_d$ (resp. $W_d$ in $W$, $W_f$ in $W$) and give a product formula in Theorem \ref{Theorem:minimal}. In \S \ref{sec:Quasi} we provide an alternative formulation of the quasi-parabolic modules $\M_f$ and then generalize these constructions to any Coxeter group containing a product of type B Weyl groups as a parabolic subgroup. We also construct a natural bar map on the induced module $\mathcal M$ and establish the canonical bases in Theorem \ref{theorem:CB}. In \S\ref{sec:4} we give explicit examples of reflection subgroups of $W(G_2)$ and $W(F_4)$ with corresponding quasi-permutation modules.

\vspace{2mm}

{\bf Acknowledgement.} 
We would like to express our great gratitude to Weiqiang Wang for his insightful advice and extremely helpful discussions. ZC's undergraduate research is supported by Wang's NSF grant (DMS-2001351). YS is supported by Graduate Research Assistantship from Wang's NSF grant and a semester fellowship from University of Virginia.


\section{Minimal length coset representatives}
\label{sec:minimallength}

Let $(W,S)$ be a Coxeter system (cf. \cite[\S 5]{H90}) consisting of a Coxeter group $W$ and a set of generators $S\subset W$. Let $W_f$ be a reflection subgroup of $W$. We denote the length function on $W$ by $l$ and the set of reflections in $W$ by $R$. Define
\begin{equation}
    S_f=\{r\in W_f\cap R\mid l(tr) > l(r),\ \forall t\in W_f\cap R, t \not= r\}
\end{equation}
Dyer \cite{Dy90} showed that \(S_f\) is a set of Coxeter generators for \(W_f\). Let $l_f$ denote the length function on $W_f$ with respect to $S_f$, then we have 
\begin{lemma}\cite{Dy90}
\label{lemma:llf}
For any \(r\in W_f\cap R, w \in W_f\),  \(l(rw) > l(w)\) if and only if \(l_f(rw) > l_f(w)\).
\end{lemma}

Then we define
\begin{equation}
\label{eq:fW}
    ^fW = \{ \sigma \in W : \text{for all } s \in S_f, l(s\sigma) > l(\sigma) \}.
\end{equation}
To study the set ${}^fW$ in general, we first provide the following lemma:
\begin{lemma}
\label{lem:length}
(i) For all \(w \in W\) and \(\sigma \in \Wf\), if \(\sigma w \in \fW\), then \(l(\sigma w) \leq l(w) - l_f(\sigma)\).

(ii) For all \(w \in \fW\) and \(\sigma \in W_f\), \(l(\sigma w) \geq l_f(\sigma) + l(w)\).
\end{lemma}
\begin{proof}
Clearly, (ii) follows from (i), so it remains to show (i). We shall proceed by an induction on \(l_f(\sigma)\). If \(l_f(\sigma) = 0\), then \(\sigma\) is the identity and the result is immediate. Otherwise, we can find \(s \in S_f\) such that \(l_f(s \sigma) < l_f(\sigma)\). Then \(l(s \sigma) < l(\sigma)\) by Lemma \ref{lemma:llf} and we also have \(l(s \sigma w) > l(\sigma w)\) by \eqref{eq:fW}.

Let \(s \sigma = t_1 t_2 \dots t_k\) and \(w = s_1 s_2 \dots s_z\) be reduced expressions with respect to the generating set \(S\). By the Strong Exchange Condition \cite[\S 5]{H90}, one of the following two cases holds: \begin{itemize}
    \item For some \(1 \leq i \leq l\), \[
        \sigma w = t_1 t_2 \dots t_k s_1 s_2 \dots \widehat{s_i} \dots s_z.
    \]
    In this case, \(l(\sigma w) = l(s \sigma w) - 1 \leq l(w) - l_f(\sigma)\) by the inductive hypothesis.
    \item For some \(1 \leq i \leq k\), \[
        \sigma w = t_1 t_2 \dots \widehat{t_i} \dots t_k s_1 s_2 \dots s_z.
    \]
    Then multiplying by \(w^{-1}\) on the left, we see that \(l(\sigma) < l(s \sigma)\), which is a contradiction.
\end{itemize}
This proves the lemma.
\end{proof}

If $S_f$ is a subset of $S$, then $W_f$ is parabolic and ${}^fW$ is the set of unique minimal length right coset representatives of $W_f$ in $W$ (see \cite{H90}). The following proposition extends the classical characterisation of ${}^fW$ from the parabolic case to the more general reflection subgroups.
\begin{proposition}
\label{Proposition:minimal}
For any reflection subgroup $W_f\subset W$, the set
\(\fW\) consists of unique minimal length right coset representatives of \(\Wf\) in \(W\).
\end{proposition}

\begin{proof}
Clearly, every minimal length right coset representative of \(\Wf\) is in \(\fW\), so it remains to show that no two distinct elements of \(\fW\) occupy the same right coset of \(\Wf\). Suppose \(w, \sigma w \in \fW\) with \(\sigma \in \Wf\). Then by Lemma \ref{lem:length} we have\begin{equation}
    \label{eq:propminimal1}
    l({\sigma w}) \leq l(w) - l_f(\sigma),
\end{equation}
but we also have\begin{equation}
    \label{eq:propminimal2}
    l({\sigma w}) \geq l(w) + l_f(\sigma).
\end{equation}
Comparing \eqref{eq:propminimal1} and \eqref{eq:propminimal2}, we must have \(l_f(\sigma) = 0\), which means \(\sigma\) is the identity.
\end{proof}

Let $\Phi$ be the root system in the geometric representation of the Coxeter system $(W,S)$. Let \(\Phi_f\) denote the roots corresponding to reflections in \(W_f\). Let $\Phi^+$ (resp. $\Phi_f^+$) denote the set of positive roots in $\Phi$ (resp. $\Phi_f$). Then we have the following lemma.
\begin{lemma}
\label{lemma:positive}
 For any \(w \in W\), the following are equivalent: \begin{enumerate}
    \item \(w \in {^f W}\),
    \item \(l(t w) > l(w)\) for any  \(t \in W_f\cap R\),
    \item \(w^{-1}(\alpha) \in \Phi^+\) for every \(\alpha \in \Phi_f^+\).
\end{enumerate}
\end{lemma}

\begin{proof}
The equivalence of (1) and (2) follows from  \eqref{eq:fW} and Lemma \ref{lem:length}. The equivalence of (2) and (3) follows from \cite[Proposition 5.7]{H90}. 
\end{proof}

Now to prepare the main theorem, we prove a `product' formula for the sets of minimal length coset representatives. Suppose we have a chain of reflection subgroups 
$W_g\subseteq W_f \subseteq W.$  Similar as in the form of \eqref{eq:fW}, we define ${^gW}$ (resp. ${^g(W_f)}$) to be the set of minimal length coset representative of $W_g$ in $W$ (resp. $W_f$).
\begin{theorem}
\label{Theorem:minimal}
For any chain of reflection subgroups
$W_g\subseteq W_f \stackrel{\text{par.}}{\subseteq} W$, we  have the following bijection:
$${^g (W_f)}\times{^f W}\longrightarrow{^g W},\qquad (\omega_1,\omega_2)\mapsto \omega_1\omega_2.$$ 
\end{theorem}

\begin{proof}
Since \(W = W_g {^g W} = W_g {^g (W_f)}{^f W}\), every right coset of \(W_g\) is represented in \({^g (W_f)} \cdot {^f W}\). Therefore, it suffices to show that \({^g (W_f)} \cdot {^f W} \subseteq {^gW}\).

 Suppose \(w_g \in {^g (W_f)}\) and \(w_f \in {^f w}\). Let $\Phi_g^+$ denote the positive roots associated to $(W_g, S_g)$. Following Lemma \ref{lemma:positive}, for any \(\alpha \in \Phi_g^+\) we have \(w_g^{-1}(\alpha) \in \Phi_f^+\) and \(w_f^{-1}w_g^{-1}(\alpha) \in \Phi^+\). Therefore, \(w_g w_f \in {^g W}\).
\end{proof}

\section{Quasi-permutation modules}
\label{sec:Quasi}
In this section we first recall the construction of quasi-permutation modules in \cite{SW21} and then generalize these constructions to any Coxeter group which contains a parabolic subgroup of type B.
\subsection{Type B quasi-permutation modules}
Let $p, q$ be two indeterminates. 
We denote $q_i=q$ for $1\leq i\leq d-1$ and $q_0=p$. The Iwahori-Hecke algebra of type B, denoted by $\HB$,  is a $\Q(p,q)$-algebra generated by $H_0,H_1,\cdots,H_{d-1}$, subject to the following relations:
\begin{align*}
&(H_i-q_i)(H_i+q_i^{-1})=0,\ \ \ \ &\text{for } i\geq 0; \\
&H_iH_{i+1}H_i=H_{i+1}H_iH_{i+1},\ \ \ \ &\text{for } i\ge 1;\\
&H_iH_j=H_jH_i,\ \ \ \ &\text{for }|i-j|>1;\\
&H_0H_1H_0H_1=H_1H_0H_1H_0. 
\end{align*}
The subalgebra generated by $H_i$, for $1\le i \le d-1$, can be identified with Hecke algebra $\Hy_{S_d}$ associated to the symmetric group $S_d$. If $\sigma \in W_d$ has a reduced expression $\sigma =s_{i_1} \cdots s_{i_k}$, we denote $H_\sigma =H_{i_1} \cdots H_{i_k}$. It is well known that $\{H_\sigma \mid \sigma \in W_d \}$ form a basis for $\HB$, and  $\{H_\sigma \mid \sigma \in S_d \}$ form a basis for $\Hy_{S_d}$. 

We recall from \cite[\S 2.3]{SW21} the construction of the quasi-permutation $\HB$-module $\M_f$. For a real number $x\in \mathbb R$ and $m \in \N$, we denote $[x, x+m] =\{x, x+1, \ldots, x+m \}$. 
For $a \in \Z_{\ge 1}$, we denote by 
\[
\I_a = \left [\frac{1-a}2, \frac{a-1}2 \right].
\]
For $r,m\in \N$ (not both zero), we introduce the following notation to indicate a fixed set partition:
\begin{equation}
  \label{eq:Ibw}
\Ibw : =\I_{2r+m}, \qquad
\Ibw =\Iwl \cup \Ib \cup \Iwr
\end{equation}
where the subsets
\begin{align}
 \label{eq:III}
\Iwr = \left [\frac{\nb+1}{2}, r+\frac{\nb-1}{2} \right].
\qquad
\Ib = \left [\frac{1-\nb}{2}, \frac{\nb-1}{2} \right], 
\qquad
\Iwl = - \Iwr,
\end{align}
have cardinalities $r, m, r$, respectively.

We view $f  \in \Ibw^d$ as a map $f: \{1, \ldots, d\} \rightarrow \Ibw$, and identify 
$f=(f(1), \ldots, f(d))$, with $f(i) \in \Ibw$. 
We define a right action of the Weyl group $W_d$ on $\Ibw^d$ such that,
for $f\in \Ibw^d$ and $0\leq j\leq d-1$,
\begin{equation}
  \label{eq:WB}
f^{s_j} =f\cdot s_j = 
\begin{cases} 
 (\cdots,f(j+1),f(j),\cdots),&\text{ if } j>0; \\
 (-f(1),f(2),\cdots,f(d)), &\text{ if } j=0,\ f(1)\in \Iwl\cup \Iwr; \\
 (f(1),f(2),\cdots,f(d)), &\text{ if } j=0,\ f(1)\in \Ib.
\end{cases} 
\end{equation} 
We sometimes write 
$$
f^\sigma = f\cdot \sigma 
=(f(\sigma(1)),\cdots,f(\sigma(d))),$$
where it is understood that 
\[
f(\sigma(i)) =
 \begin{cases}
 f(\sigma(i)), &\text{ if } \sigma(i)>0; \\
 f(-\sigma(i)), &\text{ if } \sigma(i)=0, f(-\sigma(i))\in \Ib; \\
 -f(-\sigma(i)), &\text{ if } \sigma(i)=0,f(-\sigma(i))\in \Iwl\cup \Iwr.
\end{cases}
\]

Consider the $\Q(p,q)$-vector space 
\begin{align}
\label{eq:V}
\V=\bigoplus_{a\in \Ibw}\Q(p,q)v_a. 
\end{align}
Given $f=(f(1), \ldots, f(d)) \in \Ibw^d$, we denote
\[
M_f = v_{f(1)} \otimes v_{f(2)} \otimes \ldots \otimes v_{f(d)}.
\]

\begin{lemma}\label{lem:HB} {\rm (\cite[Lemma 2.1]{SW21})}
There is a right action of the Hecke algebra $\Hy_{B_d}$ on $\V^{\otimes d}$ as follows:
$$
M_f\cdot H_i=\left\{
\begin{aligned}
&M_{f\cdot s_i}+(q-q^{-1})M_{f},\ \ & \text{ if } f(i)<f(i+1),\ i>0,\\
&M_{f\cdot s_i},\ \ &\text{ if } f(i)>f(i+1),\ i>0,\\
&qM_f,\ \ & \text{ if } f(i)=f(i+1),\ i>0, \\
&M_{f\cdot s_i}+(p-p^{-1})M_f, \ \ &\text{ if } f(1)\in \Iwr,\ i=0, \\
&M_{f\cdot s_i},\ \ &\text{ if } f(1)\in \Iwl,\ i=0, \\
&pM_f,\ \ & \text{ if } f(1)\in \Ib,\ i=0.
\end{aligned}
\right.$$
\end{lemma}

Following \cite{SW21}, let $\Ianti$ denote the subset of $\Ibw^d$ consists of $g$ where \begin{align}  \label{def:adom}
\frac{m-1}{2}\geq g(1)\geq g(2)\geq \cdots \geq g(d).
\end{align}
We can decompose $\V^{\otimes d}$ into a direct sum of cyclic submodules generated by $M_f$, for  $f\in \Ianti$, as follows:
\begin{align}
  \label{eq:decomp}
\V^{\otimes d}=\bigoplus_{f \in \Ianti} \M_f, \qquad \text{ where }\;  \M_f =M_f\HB.
\end{align}

Furthermore, we recall the bar-involution on $\M_f$. Let $f \in \Ianti$. It follows directly from Lemma \ref{lem:HB} that the quasi-permutation $\HB$-module $\M_f$ has a standard basis $\{M_{f\cdot\sigma}\mid \sigma\in {}^fW_d\}$. We define a $\Q$-linear map $\iba$ on the module $\M_f$ by
\begin{align}
  \label{eq:barM}
\iba(q)=q^{-1},\quad \iba(p)=p^{-1},\quad \iba(M_{f \cdot \sigma})=M_f\bar H_{\sigma},\quad \forall \sigma \in {}^fW.
\end{align}
\begin{proposition} ({\rm \cite[Proposition 3.12]{SW21}})
  \label{prop:iHba}
Let $f \in \Ianti,$. The map $\iba$ on $\M_f$ in \eqref{eq:barM} is compatible with the bar operator on the Hecke algebra, i.e., 
\begin{align}
  \label{MxH}
\iba(xh)=\iba(x) \overline{h}, 
\qquad
\text{for all $x \in \M_f, \  h \in \HB$}.
\end{align}
In particular, $\iba^2 =\text{Id}$. (We shall call $\iba$ the bar involution on $\M_f$.)
\end{proposition}

\subsection{3-step induction}
\label{sec:3step}
While the module $\M_f$ is not an induced module in general, we outline in this subsection an alternative construction (up to an isomorphism) by a 3-step induction process. This 3-step induction process was proposed in \cite[Remark 3.13]{SW21}.

We fix an $f \in \Ianti$ of the form
\begin{align}
  \label{eq:f}
f= 
(\underbrace{a_1,\ldots,a_1}_{m_1},\ldots,\underbrace{a_k,\ldots,a_k}_{m_k},\underbrace{a_{k+1},\ldots,a_{k+1}}_{m_{k+1}},\ldots,\underbrace{a_l,\ldots,a_l}_{m_l}),
\end{align}
where $a_1> \ldots >a_k>a_{k+1}> \ldots > a_l,\ \{a_1,\ldots,a_k\}\subset \Ib,\ \{a_{k+1},\ldots, a_l\}\subset \Iwl$, and $m_1+\ldots+m_l=d$. The stabilizer subgroup of $f$ in $W_d$ is 
\begin{align}
  \label{eq:Wf}
W_f =W_{{m_1}}\times \ldots \times W_{{m_k}}\times S_{m_{k+1}}\times \ldots\times S_{m_l}. \end{align} 

Denote 
\begin{align*}
S_f &:= S_{{m_1}}\times \ldots \times S_{{m_k}} \times S_{m_{k+1}}\times \ldots\times S_{m_l},
\\
W_f^\bu &:= W_{d_\bu} \times S_{m_{k+1}}\times \ldots\times S_{m_l},
\quad
S_f^\bu := S_{d_\bu} \times S_{m_{k+1}}\times \ldots\times S_{m_l}.
\end{align*}
For any reflection group $W$, we denote the Hecke algebra associated to it by $\Hy(W)$. For any two reflection groups $W_1$ and $W_2$ we say $W_1\stackrel{\text{par.}}{\subset}W_2$ if $W_1$ is a parabolic subgroup of $W_2$.

Clearly, $W_f$ is a (not necessarily parabolic) reflection subgroup of $W_f^\bu$, and $W_f^\bu\stackrel{\text{par.}}{\subset}W_d$, and hence, the Hecke algebra $\Hy(W_f^\bu)$ can be viewed naturally as a subalgebra of $\HB$. Denote by ${}^fW^\bu$ the set of minimal length representatives of $W_f^\bu \backslash W_d$.

The set of minimal length representatives of $W_f \backslash W_f^\bu$ can be naturally identified with the set of minimal length representatives of $S_f \backslash S_f^\bu$, which will be denoted by ${}^f\mathcal D$. Since $S_f\stackrel{\text{par.}}{\subset}S_f^\bu$,  the Hecke algebra $\Hy(S_f)$ is naturally a subalgebra of the Hecke algebra $\Hy(S_f^\bu)$. Denote by $M_{tr}$ the trivial representation of $\Hy(S_f)$. We get an induced $\Hy(S_f^\bu)$-module: 
\[
\cM_f^\bu :=M_{tr} \otimes_{\Hy(S_f)} \Hy(S_f^\bu).
\]

 We now extend the $\Hy(S_f^\bu)$-action on $\cM_f^\bu$ to an $\Hy(W_f^\bu)$-action by letting $H_0$ act as $p\cdot \text{Id}$, one can and only needs to check that the braid relation $H_0H_1H_0H_1 =H_1H_0H_1H_0$ trivially holds. 

Now, as $W_f^\bu\stackrel{\text{par.}}{\subset}W_d$, we can define an induced $\HB$-module
\[
\cM_f := \cM_f^\bu \otimes_{\Hy(W_f^\bu)} \HB. 
\]
Then $\cM_f$ has a $\Q(q)$-basis given by $\{(1\otimes H_{\sigma_1}) \otimes H_{\sigma_2}\mid\sigma_1 \in {}^f\mathcal D, \sigma_2 \in {}^fW^\bu\}$. 

The set ${}^fW$ of minimal length representatives of $W_f \backslash W_d$ can be naturally identified as ${}^fW = {}^f\mathcal D \cdot {}^fW^\bu$, cf. \cite[Theorem 2.2.5]{DS00}. One can then show that there exists a natural $\HB$-module isomorphism 
\[
\cM_f \stackrel{\cong}{\longrightarrow} \M_f,
\qquad (1\otimes H_{\sigma_1}) \otimes H_{\sigma_2} \mapsto M_{f \cdot \sigma_1\sigma_2}.
\]

One advantage of this $3$-step induction construction is that the bar involution on  $\cM_f$ (compatible with the bar map on $\HB$) is directly visible from the induced module structure.

\subsection{Induced module $\mathcal M$}
Motivated by the 3-step induction procedure, we can generalize it to general Coxeter groups which contain a product of type B Weyl groups as a parabolic subgroup.

Suppose we have the following chain of reflection groups
$$W_f\subset W_{d}\stackrel{\text{par.}}{\subset} W,$$
where $W$ is any Coxeter group with a generating set $S$ and $\{s_0,s_1,\cdots,s_{d-1}\}\subset S$.

The Hecke algebra $\Hy=\Hy(W,S)$ associated to $(W,S)$ is a $\Q(\{q_s\mid s\in S\})$-algebra generated by $\{H_{s}\mid s\in S\}$ subject to the Braid group relations and the quadratic relations:
\begin{align*}
&(H_s-q_s)(H_s+q_s^{-1})=0,\ \ \ \ &\text{for } s\in S.
\end{align*}
Note that $\{q_s\mid s\in S \}$ is a set of parameters where $q_s=q_t$ if $s,t\in S$ are conjugate in $W$.

\begin{lemma} 
Every right coset of $W_f$ in $W_d$ has a unique minimal length representative.
\end{lemma}

\begin{proof}
See Proposition \ref{Proposition:minimal}.
\end{proof}

Let ${}^f(W_d)$ (resp. ${}^dW$) denote the set of minimal length right coset representatives for $W_f$ (resp. ${W_d}$) in $W_d$ (resp. $W$), then we have the following corollary according to Theorem \ref{Theorem:minimal}:
\begin{corollary}
\label{cor:fW}
There is a bijection $${}^fW_d\times {}^dW\longrightarrow {}^fW,\qquad (\omega_1,\omega_2)\mapsto \omega_1\cdot \omega_2.$$
\end{corollary}

It follows directly from Lemma \ref{lem:HB} that the quasi-permutation $\HB$-module $\M_f$ has a standard basis $\{M_{f\cdot\sigma}\mid \sigma\in {}^fW_d\}$. For the sake of readability we define $M^f_{ \omega}=M_{f\cdot \omega}$ for any $\omega \in W_d$ and rewrite the $\HB$-action on $\M_f$ as follows (cf. \cite[Proposition 3.8]{SW21}):
\begin{equation}
\label{equa:typeB act}
\begin{aligned}
M^f_{ \sigma} H_i =
\begin{cases}
M^f_{\sigma s_i} +(q_i -q_i^{-1}) M^f_{\sigma}, & \text{ if }\ \sigma s_i<\sigma,
\\
M^f_{ \sigma s_i}, & \text{ if }\ \sigma s_i> \sigma \text{ and } \sigma s_i \in {}^fW_d,
\\
q_iM^f_{\sigma}, & \text{ if } \  i\neq 0,\ \sigma s_i> \sigma \text{ and } \sigma s_i \not\in {}^fW_d,
\end{cases}
\end{aligned}
\end{equation}
for $i=0,1,\ldots,d-1$.

Then we introduce an induced module
$$\mathcal M=\M_f\otimes_{\HB}\Hy.$$
We still call $\mathcal M$ a  quasi-permutation module. We see that the induced module $\mathcal M$ has a natural standard basis $\{m_{\sigma x}=M^f_{\sigma}\otimes H_x \mid \sigma\in {}^fW_d,x\in {}^dW\}$.

\begin{proposition}
\label{prop:HM}
The action of $\Hy$ on $\mathcal M$ is given by 
\begin{align*}
m_{\sigma x}H_s =
\begin{cases}
m_{\sigma xs} +(q_s -q_s^{-1}) m_{\sigma x}, & \text{ if }\ \sigma x s<\sigma x,
\\
m_{\sigma xs}, & \text{ if }\ \sigma xs> \sigma x \text{ and } \sigma xs \in {}^fW,
\\
q_sm_{\sigma x}, & \text{ if } \ \sigma x s> \sigma x \text{ and } \sigma x s \not\in {}^fW,
\end{cases}
\end{align*}
for any $s\in S$.
\end{proposition}

\begin{proof}
For any $\sigma \in {}^fW_d,x\in {}^dW,s\in S$, we have several possible cases here:

(i) If $xs<x$, then we must have $xs \in {}^dW$. Thus $\sigma x s<\sigma x$ and $\sigma x s\in {}^fW$. In this case we have
$H_xH_s=H_{xs}H_s^2=H_{xs}+(q_s-q_s^{-1})H_x$. Therefore
$$m_{\sigma x}H_s=M^f_{ \sigma}\otimes H_x H_s=M^f_{\sigma }\otimes H_{xs} +(q_s -q_s^{-1}) M^f_{\sigma}\otimes H_x=m_{\sigma xs} +(q_s -q_s^{-1}) m_{\sigma x}.$$

(ii) If $xs>x$ and $xs\in {}^dW$, then we have $\sigma xs\in {}^fW$ and
$$m_{\sigma x}H_s=M^f_{ \sigma}\otimes H_x H_s=M^f_{\sigma }\otimes H_{xs}=m_{\sigma xs}.$$

(iii) If $xs>x$ and $xs\notin {}^dW$, then we must have $xs=s_jx$ for some $j\in \{0,1,\ldots,d-1\}$. 

(iii-a) If $l(\sigma s_j)<l(\sigma)$, we have $\sigma s_j\in {}^f W_d$. Thus we have $\sigma s_jx=\sigma xs<\sigma x$ and 
\begin{align*}
m_{\sigma x}H_s=&M^f_{ \sigma}\otimes H_x H_s=M^f_{ \sigma}\otimes H_j H_x=M^f_{ \sigma}H_j\otimes H_x\\
=&M^f_{ \sigma s_j}\otimes H_x+(q_j-q_j^{-1})M^f_{ \sigma}\otimes H_x \\
=&m_{\sigma s_jx}+(q_j-q_j^{-1})m_{ \sigma x}=m_{\sigma x s}+(q_j-q_j^{-1})m_{ \sigma x}.
\end{align*}
Since $s$ and $s_j$ are conjugate to each other, we have $q_j=q_s$.

 (iii-b) If $l(\sigma s_j)>l(\sigma)$ and $\sigma s_j\in {}^fW_d$, then we have $\sigma xs>\sigma x,\ \sigma xs \in {}^fW$ and
\begin{align*}m_{\sigma x}H_s =&M^f_{ \sigma}\otimes H_x H_s=M^f_{ \sigma}\otimes H_j H_x=M^f_{ \sigma}H_j\otimes H_x=M^f_{ \sigma s_j}\otimes H_x\\
=&m_{\sigma s_jx}=m_{\sigma x s}.
\end{align*}

(iii-c) If $l(\sigma s_j)>l(\sigma)$ and $\sigma s_j\notin {}^fW_d$, then there are two possible cases according to \cite[Theorem 3.6 $(iii),(iii_0)$]{SW21}, in either cases we have that $\sigma s_j\in W_f \sigma$ and hence the unique minmal length representative of $W_f \sigma x s$ is $\sigma x$. 

Thus $\sigma x s\notin {}^fW$ and thanks to \eqref{equa:typeB act} we have
$$m_{\sigma x}H_s=M^f_{ \sigma}\otimes H_x H_s=M^f_{ \sigma}\otimes H_j H_x=M^f_{ \sigma}H_j\otimes H_x=q_jHm_{\sigma x}.$$
Again we have $q_j=q_s$ by the conjugacy of $s$ and $s_j$ in $W$.
\end{proof}

\subsection{Canonical basis on $\mathcal M$}
Now we define a bar-involution on $\mathcal M$ by 
$${}^{-}\colon m\otimes h \mapsto \iba (m) \otimes \bar h,\quad \forall m\in \M_f,\ h\in \Hy,$$
where $\iba$ is from Proposition \ref{prop:iHba} and it is compatible with the bar involution on $\HB$. Thus above bar map on $\mathcal M$ is well-defined.

For the formulation of canonical basis on $\M_f$, we shall specialize to a one-parameter setting. Suppose $q_s\in q^{\Z}$ for all $s\in S$. Then $\Hy$ becomes a $\Q(q)$-algebra. The bar involution $'\overline{\ \cdot\ }'$ remains valid.
\begin{theorem}
\label{theorem:CB}
Suppose $q_s\in q^{\Z}$ for all $s\in S$, then for each $w\in{}^fW$, there exists a unique element $C_w\in \mathcal M$ such that
\begin{enumerate}
\item[(i)]
$\overline{ C_{w}} =C_{w}$,
\item[(ii)]
$C_{w} \in m_w+\sum_{y\in {}^fW }\limits 
q^{-1}\Z[q^{-1}] m_y.$
\end{enumerate}
Moreover we have 	
\begin{enumerate}
\item[(ii$'$)]  $C_{w} \in m_{w} +\sum_{y\in {}^fW, y< w}\limits 
q^{-1}\Z[q^{-1}] m_y.$
\end{enumerate}
\end{theorem}

The set $\{ C_{\sigma} | \sigma \in {}^fW \}$ is called a {\em canonical basis} of $\mathcal M$. 

\begin{proof}
Let $w=\sigma x$ where $\sigma \in {}^fW_d,\ x\in {}^dW$. 

For those $s\in S$ such that $q_s\in q^{\Z_{>0}}$, we set $b_s=H_s+q_s^{-1}$, which is bar invariant. Now Proposition \ref{prop:HM} can be rewritten as
\begin{align} \label{eq:MM1}
m_{\sigma x}H_s=
\begin{cases}
m_{\sigma xs} +q_s m_{\sigma x}, & \text{ if }\ \sigma x s<\sigma x,
\\
m_{\sigma xs}+q^{-1}_s m_{\sigma x}, & \text{ if }\ \sigma xs> \sigma x \text{ and } \sigma xs \in {}^fW,
\\
(q_s+q_s^{-1})m_{\sigma x}, & \text{ if } \ \sigma x s> \sigma x \text{ and } \sigma x s \not\in {}^fW.
\end{cases}
\end{align}

For those $s\in S$ such that $q_s\in q^{\Z_{<0}}$, we set $b_s=H_s-q_s$, which is also bar invariant. Now Proposition \ref{prop:HM} can be rewritten as
\begin{align} \label{eq:MM2}
m_{\sigma x}H_s =
\begin{cases}
m_{\sigma xs} -q^{-1}_s m_{\sigma x}, & \text{ if }\ \sigma x s<\sigma x,
\\
m_{\sigma xs}-q_s m_{\sigma x}, & \text{ if }\ \sigma xs> \sigma x \text{ and } \sigma xs \in {}^fW,
\\
0, & \text{ if } \ \sigma x s> \sigma x \text{ and } \sigma x s \not\in {}^fW.
\end{cases}
\end{align}

Now the existence of $C_{\sigma}$ satisfying Conditions~(i) and (ii$'$) can be proved using \eqref{eq:MM1} and \eqref{eq:MM2} by an induction on the Chevalley-Bruhat order for $\sigma$, following exactly the same argument as for \cite[Theorem 3.1]{So97}. 
\end{proof}

\begin{remark}
If we consider disjoint type $B$ subgroups $W_{d_i}\subset W$ and quasi-parobolic subgroups $W_{f_i}\subset W_{d_i}$, then we can easily generate this construction to the level of $$\cdots\times W_{f_i}\times \cdots\subset \cdots\times W_{d_i}\times \cdots\subset W.$$
\end{remark}

\section{ Quasi-parabolic modules in type $G_2$ and $F_4$}
\label{sec:4}
In this section we give explicit examples of {\bf non-parabolic} reflection subgroups of Coxeter group of different types and construct the corresponding quasi-permutation modules.

We often use the shorthand notation $s_{ijk\cdots} =s_i s_j s_k \cdots$ below.

\subsection{Type $G_2$}
\label{sec:G2}
Let $W=W(G_2)$ be the Coxeter group of type $G_2$. The simple roots of $W(G_2)$ are denoted by $\alpha_1, \alpha_2$, with $\alpha_1$ being the long simple root.
 The Hecke algebra $\Hy =\Hy(G_2)$ of type $G_2$ is generated by $H_1, H_2$ subject to relations:
  \begin{align*}
    (H_1+ p^{-1})(H_1-p )&=0,  \quad 
    (H_2+ q^{-1})(H_2-q )=0,  \\ 
    \text{(braid group relation) \quad} (H_1H_2)^3&=(H_2H_1)^3.
  \end{align*}

\subsubsection{Subgroup of type $A_2$}
\label{subsub:G2A2}

 We consider $W_f =\langle s_2, s_{121} \rangle =\{e, s_2, s_{121}, s_{1212}, s_{2121}, s_{21212} \}$, with ${}^fW =\{e, s_1\}.$ Then we consider the 2-dimensional vector space 
  \[
  \M_f =\text{span} \{M_f, M_{f\cdot s_1} \}.
  \]
  
\begin{proposition}
$\M_f$ is a right $\Hy$-module given by
\begin{align*}
M_f H_1 =M_{f\cdot s_1},& \quad
M_{f\cdot s_1} H_1 =  (p -p^{-1}) M_{f\cdot s_1} +M_f, 
\\
M_f H_2 =q M_f,& \quad
M_{f\cdot s_1} H_2 = q M_{f\cdot s_1}.
\end{align*}
\end{proposition}

\begin{proof}
Since $H_2$ acts on $\M_f$ as $q\cdot \text{Id}$, clearly $H_1, H_2$ satisfy the braid relation for $G_2$. 
\end{proof}

\subsubsection{Subgroup of type $A_1\times A_1$}

We consider $W_f =\langle s_2, s_{12121} \rangle =\{e, s_2, s_{12121}, s_{121212}\}$,  ${}^fW =\{e, s_1, s_1s_2\}.$ Then we consider the 3-dimensional vector space 
  \[
  \M_f =\text{span} \{M_f, M_{f\cdot s_1}, M_{f\cdot s_1s_2} \}.
  \]

\begin{proposition}
\label{prop:d3}
$\M_f$ is a right $\Hy$-module given by
\begin{align*}
M_f H_1 =M_{f\cdot s_1}, \quad
M_{f\cdot s_1} H_1 &= (p -p^{-1}) M_{f\cdot s_1} +M_f, \quad
M_{f\cdot s_1s_2} H_1 = p M_{f\cdot s_1s_2}, 
\\
M_f H_2 =q M_f, \quad
M_{f\cdot s_1} H_2 &=  M_{f\cdot s_1s_2}, \quad
M_{f\cdot s_1s_2} H_2 = (q-q^{-1}) M_{f\cdot s_1s_2}+ M_{f\cdot s_1}.
\end{align*}
\end{proposition}

\begin{proof}
We check the braid relation holds by computer.
\end{proof}

\subsubsection{Subgroup of type $A_1\times A_1$, again}

We consider $W_f =\langle s_1, s_{21212} \rangle =\{e, s_1, s_{21212}, s_{121212}\}$,  ${}^fW =\{e, s_2, s_2 s_1\}.$ 
Then we consider the 3-dimensional vector space 
  \[
  \M_f =\text{span} \{M_f, M_{f\cdot s_2}, M_{f\cdot s_2s_1} \}.
  \]

\begin{proposition}
$\M_f$ is a right $\Hy$-module given by
\begin{align*}
M_f H_1 = pM_{f}, \quad
M_{f\cdot s_2} H_1 &= M_{f\cdot s_2s_1}, \quad
M_{f\cdot s_2s_1} H_1 = (p-p^{-1}) M_{f\cdot s_2s_1} +M_{f\cdot s_2}, 
\\
M_f H_2 = M_{f\cdot s_2}, \quad
M_{f\cdot s_2} H_2 &=  (q-q^{-1}) M_{f\cdot s_2} + M_{f}, \quad
M_{f\cdot s_2s_1} H_2 = q M_{f\cdot s_2s_1}.
\end{align*}
\end{proposition}

\begin{proof}
We check the braid relation holds by computer.
\end{proof}

\subsection{Type $F_4$}
\label{sec:F4}
Let $W=W(F_4)$ be the Coxeter group of type $F_4$. The simple roots of $W(F_4)$ are denoted by $\alpha_1 =\e_1-\e_2, \alpha_2 =\e_2 -\e_3$ (long) and $\alpha_3 =\e_3, \alpha_4 =\hf(\e_4-\e_1-\e_2-\e_3)$ (short). Recall $|W| =2^7 3^2$. 

 Accordingly, the Hecke algebra $\Hy =\Hy(F_4)$ of type $F_4$ is generated by $H_1, H_2, H_3, H_4$ subject to relations:
  \begin{align*}
    (H_i + p^{-1})(H_i -p)&=0 \; (i=1,2),  \quad 
    (H_j + q^{-1})(H_j -q)=0 \; (j=3, 4),  \\ 
    \text{(braid group relations) \quad} 
    H_1H_2H_1&=H_2H_1H_2, 
    \quad
     H_2H_3H_2H_3 =H_3H_2H_3H_2, \\
     H_3 H_4 H_3 &=H_4 H_3 H_4, \quad
     H_iH_j =H_j H_i \; (|i-j| \ge 2).
  \end{align*}
  
  \subsubsection{Subgroup of type $C_4$}
  \label{sub:C4}


Let $W_f =\langle s_2, s_3, s_4, s_{12321} \rangle$, and  ${}^fW=\{e, s_1, s_1s_2\}$. 
Then we consider the 3-dimensional vector space 
  \[
  \M_f =\text{span} \{M_f, M_{f\cdot s_1}, M_{f\cdot s_1s_2} \}.
  \]

\begin{proposition}
\label{prop:F4C3}
$\M_f$ is a right $\Hy$-module given by
\begin{align*}
M_f H_1 =M_{f\cdot s_1}, \quad
M_{f\cdot s_1} H_1 &= (p -p^{-1}) M_{f\cdot s_1} +M_f, \quad
M_{f\cdot s_1s_2} H_1 = p M_{f\cdot s_1s_2}, 
\\
M_f H_2 =p M_f, \quad
M_{f\cdot s_1} H_2 &=  M_{f\cdot s_1s_2}, \quad
M_{f\cdot s_1s_2} H_2 = (p-p^{-1}) M_{f\cdot s_1s_2}+ M_{f\cdot s_1},
\\
M_f H_3 =q M_f, \quad
M_{f\cdot s_1} H_3 &=  qM_{f\cdot s_1}, \quad
M_{f\cdot s_1s_2} H_3 = q M_{f\cdot s_1s_2},
\\
M_f H_4 =q M_f, \quad
M_{f\cdot s_1} H_4 &=  qM_{f\cdot s_1}, \quad
M_{f\cdot s_1s_2} H_4 = q M_{f\cdot s_1s_2}.
\end{align*}
\end{proposition}

\begin{proof}
Thanks to $H_3 =q\cdot \Id$ and $H_4 =q\cdot \Id$, we only need to check the braid relation between $H_1, H_2$. In this case, it follows by the same checking as for the case of the parabolic subgroup $\langle s_2 \rangle$ in $\langle s_1, s_2 \rangle$. 
\end{proof}

\subsubsection{}
\label{sec:F4fail}
We consider $W_g =\langle s_3, s_4, s_{12321}  \rangle$; this is a parabolic subgroup of the $W_f$ in \S\ref{sub:C4}. Hence we have $W_g \stackrel{\text{par.}}{\subset} W_f \subset W$. 

Note that $|{}^gW| =48$, as $|W_g \backslash W_f| =16$ and $|{}^fW| =3$. We can check ${}^gW={}^gW_f\cdot {}^fW$. Following the general construction we can define a vector space $\M_g$ parameterized by ${}^gW$.

We check by computer that the $\Hy$-actions on $\M_g$ in this case does not define a module structure, i.e. the braid relations fail to hold. 
\subsubsection{Subgroup of type $B_3 A_1$}
\label{F4:B3A1}

Let $W_f =\langle s_1, s_2, s_3, s_{432312343231234} \rangle$, $|W_f| = 2^5 3$, $|{}^fW|=12$, and   
\[
{}^fW =\{e, s_4, s_{43}, s_{432}, s_{4321},  s_{4323},  s_{43231}, s_{43234}, s_{432312}, s_{432341}, s_{4323123}, s_{4323412} \}. 
\]

In this case we also check by computer that the $\Hy$-actions on $\M_f$ does not define a module structure, i.e. the braid relations fail to hold.



\begin{thebibliography}{ABCD90}

\bibitem[De87]{De87} V.~Deodhar, 
{\em On Some Geometric Aspects of Bruhat Orderings II. The Parabolic Analogue of Kazhdan-Lusztig Polynomials}, J. Algebra {\bf 111} (1987), 483--506.

\bibitem[DPR13]{DPR13} J.M.~ Douglass, G.~Pfeiffer, and G. R\"ohrle,  
{\em On reflection subgroups of finite Coxeter groups},
 Comm. Algebra {\bf 41} (2013), 2574--2592.  

\bibitem[DS00]{DS00} J.~Du and L.~Scott,
{\em The $q$-Schur${^2}$ algebra}, Trans. Amer. Math. Soc. {\bf 352} (2000), 4325--4353.

\bibitem[Dy90]{Dy90} M.~Dyer,
{\em Reflection Subgroups of Coxeter Systems}, J. Algebra {\bf 135} (1990), 57--73.


\bibitem[H90]{H90}
J.~Humphreys, {\em Reflection groups and Coxeter groups}, Cambridge Studies in Advanced Mathematics, {\bf 29}, Cambridge University Press, 1990.


\bibitem[KL79]{KL79} D. Kazhdan and G. Lusztig,
{\em Representations of Coxeter groups and Hecke algebras},
Invent. Math. {\bf 53} (1979), 165--184.

\bibitem[Lus84]{Lus84}
G. Lusztig,
{\em Characters of reductive groups over a finite field}, Annals of Mathematics Studies {\bf 107}, Princeton University Press, 1984. 

\bibitem[Lus03]{Lus03}
G. Lusztig,
{\em Hecke algebras with unequal parameters},
CRM Monograph Series {\bf 18}, Amer. Math. Soc., Providence, RI, 2003; for an updated and enlarged version see \href{https://arxiv.org/abs/math/0208154}{arXiv:0208154v2}

\bibitem[LW20]{LW20}
L.~Luo and W.~Wang, {\em Lectures on dualities ABC in representation theory}, ''Forty Years of Algebraic Groups, Algebraic Geometry, and Representation Theory in China'', (eds. J. Du, J. Wang, L. Lin), World scientific, 2022.
\href{https://arxiv.org/abs/2012.07203}{arXiv:2012.07203}


%

\bibitem[So97]{So97} W.~Soergel,
{\em Kazhdan-Lusztig polynomials and a combinatoric for tilting modules}, Represent. Theory {\bf1} (1997), 83--114. 


\bibitem[SW21]{SW21} Y.~Shen and W.~Wang,
{\em $\imath$Schur duality and Kazhdan-Lusztig basis expanded}, to appear on Adv. in Math.
\href{https://arxiv.org/abs/2108.00630}{arXiv:2108.00630}


 \end{thebibliography}
\end{document}